\documentclass{amsart}

\usepackage[utf8]{inputenc}

\usepackage{amsmath}
\usepackage{amssymb}
\usepackage{verbatim}
\usepackage{graphicx}

\usepackage{tikz-cd}
\usepackage{geometry}
\usepackage{esint}
\usepackage[colorlinks=true,urlcolor=blue,citecolor=red,linkcolor=blue,linktocpage,pdfpagelabels,bookmarksnumbered,bookmarksopen]{hyperref}
\numberwithin{equation}{section}

\usepackage[hyperpageref]{backref}


\date{\today}
\newtheorem{lemma}{Lemma}[section]
\newtheorem{theorem}[lemma]{Theorem}
\newtheorem{proposition}[lemma]{Proposition}
\newtheorem{corollary}[lemma]{Corollary}
\theoremstyle{definition}
\newtheorem{definition}[lemma]{Definition}
\newtheorem{remark}[lemma]{Remark}
\newtheorem*{ack}{Acknowledgments}

\title{On fractional Hardy-type inequalities in general open sets }

\author[Cinti]{Eleonora Cinti}

\address[E. Cinti]{Dipartimento di Matematica\newline\indent 
	Alma Mater Studiorum Universit\`a di Bologna
	\newline\indent
	piazza di Porta San Donato, 5\newline\indent
	40126 Bologna - Italy}
\email{eleonora.cinti5@unibo.it}

\author[Prinari]{Francesca Prinari}

\address[F. Prinari]{Dipartimento di Scienze Agrarie, Alimentari e Agro-ambientali
\newline\indent 
Universit\`a di Pisa
\newline\indent
Via del Borghetto 80, 56124 Pisa, Italy}
\email{francesca.prinari@unipi.it}

\begin{document}

\begin{abstract}  We show that,  when  $sp>N$,   the sharp Hardy constant  $\mathfrak{h}_{s,p}$ of the punctured space $\mathbb R^N\setminus\{0\}$ in  the  Sobolev-Slobodecki\u{\i} space provides 
  an optimal lower bound for the Hardy constant $\mathfrak{h}_{s,p}(\Omega)$ of an open  $\Omega\subsetneq \mathbb R^N$.
 The proof exploits the characterization of Hardy’s inequality in the fractional setting  in
terms of positive local weak supersolutions of the relevant Euler-Lagrange equation and relies  on the construction of suitable supersolutions by means of the distance function from the boundary of $\Omega$. 
Moreover, we compute the limit  of   $\mathfrak{h}_{s,p}$  as $s\nearrow 1$,   as well as the limit when   $p \nearrow \infty$. Finally, we apply our results to establish a lower bound for the non-local eigenvalue $\lambda_{s,p}(\Omega)$ in terms of $\mathfrak{h}_{s,p}$ when $sp>N$, which, in turn, gives an improved Cheeger inequality whose constant does not vanish as $p\nearrow \infty$. \end{abstract}

\maketitle

\begin{center}
\begin{minipage}{10cm}
\small
\tableofcontents
\end{minipage}
\end{center}

\section{Introduction}
This paper deals with Hardy-type inequalities in fractional Sobolev spaces, with special interest to optimal lower bounds on their sharp contants.

We recall some well known facts on Hardy's inequalities in the classical (local) setting.
For an open subset $\Omega$ of $\mathbb R^N$, let us define the distance function to the boundary as
$$
d_\Omega(x):=\min_{y\in \partial \Omega}|x-y|,\quad \mbox{for all } x\in \Omega.
$$
A classical result in the theory of Sobolev spaces states that, under suitable assumptions on the set $\Omega$, there exists a positive constant $C$ such that 
\begin{equation}\label{classical-Hardy}
	C\int_\Omega \frac{|u|^p}{d_\Omega^p}\, dx\leq \int_{\Omega} |\nabla u|^p\,dx, \qquad \text{for all } u\in C_0^\infty(\Omega).
 	\end{equation}
 In the huge existing literature concerning the Hardy inequality, some results establish  necessary and sufficient conditions on the open set $\Omega$ ensuring  the validity of  \eqref{classical-Hardy} (see, for example, the references in \cite{KK}) while other papers are devoted to  the  interesting related question  of  determining   the optimal constant $C$ in \eqref{classical-Hardy}, which can be defined in a variational way as
\[
\mathfrak{h}_{p}(\Omega)=\inf_{u\in C^\infty_0(\Omega)}\left\{\int_\Omega |\nabla u|^p :\, \int_{\Omega} \frac{|u|^p}{d_{\Omega}^{p}} \, dx=1\right\}.
\]
This has been achieved in some particular cases, e.g. :
\begin{itemize}
	\item if $\Omega=\mathbb R^N\setminus\{0\}$,  $1<p<\infty$,  $p\not=N$, then $\mathfrak{h}_{p}(\Omega)=\left(\frac{|N-p|}{p}\right)^p$ (see \cite{PK} and the references therein); 
	\item if $\Omega$ is convex, $1<p<\infty$, then $\mathfrak{h}_{p}(\Omega)=\left(\frac{p-1}{p}\right)^p$ (for a proof, see  \cite[Theorem 11]{MMP}). \end{itemize}

In the particular case  $p>N$, Lewis in \cite{Lewis} and Wannebo in \cite{Wa} show that the Hardy inequality \eqref{classical-Hardy} holds on every open set $\Omega\subset \mathbb R^N$.  Later, an alternative proof of this  result  has been given  in  \cite{H} by means of a ''pointwise Hardy inequality"  and maximal function techniques. However, all these papers do not  provide any explicit (lower) bound for $\mathfrak{h}_{p}(\Omega)$. The latter question is  studied in  \cite{Av, GPP},  where it is proved that,  when $p>N$, the optimal Hardy constant of the punctured space  $\mathbb R^N\setminus\{0\}$ provides an optimal lower bound for $\mathfrak{h}_{p}(\Omega)$, i.e. for every open set $\Omega\subset  \mathbb{R}^N$ it holds:
\begin{equation}\label{s=1} \mathfrak{h}_{p}(\Omega)\geq \mathfrak{h}_{p}(\mathbb{R}^N\setminus\{0\})=\left(\frac{p-N}{p}\right)^p.\end{equation}

Recently, much interest has been devoted to the study of fractional nonlocal operators, fractional Sobolev spaces and related functional inequalities. A natural question in this context is whether a fractional analogue of \eqref{classical-Hardy} holds true and whether one can determine the sharp constant, at least in some particular cases.

 In order to state our main result, let us start by introducing our notation.
For $1\le p<\infty$, $0<s\le 1$ and $\Omega\subseteq  \mathbb{R}^N$, we define
\[
W^{s,p}(\Omega)=\Big\{\varphi\in L^p(\Omega)\, :\, [\varphi]_{W^{s,p}(\Omega)}<+\infty\Big\},
\]
where
\[
[\varphi]_{W^{s,p}(\Omega)}=\left\{\begin{array}{ll}
\displaystyle \left(\iint_{\Omega\times \Omega} \frac{|\varphi(x)-\varphi(y)|^p}{|x-y|^{N+sp}}\,dx\,dy\right)^\frac{1}{p},& \mbox{ if } 0<s<1,\\
&\\
\displaystyle\left(\int_{\Omega} |\nabla \varphi|^p\,dx\right)^\frac{1}{p},& \mbox{ if } s=1. 
\end{array}
\right.
\]
When $1<p<\infty$, this is a reflexive space, when endowed with the norm
\[
\|\varphi\|_{W^{s,p}(\Omega)}=\|\varphi\|_{L^p(\Omega)}+[\varphi]_{W^{s,p}(\Omega)},\qquad \mbox{ for every } \varphi\in W^{s,p}(\Omega).
\]
We also indicate by $\widetilde{W}^{s,p}_0(\Omega)$ the closure of $C^\infty_0(\Omega)$ in $W^{s,p}(\mathbb{R}^N)$. 
\par
The analogue of the Hardy inequality \eqref{classical-Hardy} in this context reads as follows:
\begin{equation}\label{fractional-Hardy}
	C\int_\Omega \frac{|u|^p}{d_\Omega^{sp}}\, dx\leq  [u]_{W^{s,p}(\mathbb R^N)}^p, \quad \mbox{for all } u \in C_0^{\infty}(\Omega).
\end{equation}
For an open set  $\Omega\subsetneq \mathbb{R}^N$,  we introduce  its sharp fractional $(s,p)$-Hardy constant defined as  
\[
\mathfrak{h}_{s,p}(\Omega)=\inf_{u\in C^\infty_0(\Omega)}\left\{[u]^p_{W^{s,p}(\mathbb{R}^N)} :\, \int_{\Omega} \frac{|u|^p}{d_{\Omega}^{sp}} \, dx=1\right\}.
\]
We explicitly note that  $\mathfrak{h}_{1,p}(\Omega)=\mathfrak{h}_{p}(\Omega)$.
 Observe that, by definition of  $\widetilde{W}^{s,p}_0(\Omega)$, we have
	\begin{equation}\label{inf-W}
	\mathfrak{h}_{s,p}(\Omega)=\inf_{u\in \widetilde{W}^{s,p}_0(\Omega)}\left\{[u]^p_{W^{s,p}(\mathbb{R}^N)} :\, \int_{\Omega} \frac{|u|^p}{d_{\Omega}^{sp}} \, dx=1\right\},
	\end{equation}
by a standard density argument.	
\par	
	A first result in the determination of the sharp constant in this fractional setting was established by Frank and Seiringer \cite[Theorem 1.1]{FS}, who proved that, if $N\ge 1$, $0<s<1$ and $1\le p<\infty$ are such that $sp\not=N$, then
	\[
	 \mathfrak{h}_{s,p}:=\mathfrak{h}_{s,p}(\mathbb{R}^N\setminus\{0\})=2\,\int_0^1 r^{sp-1}\,\left|1-r^\frac{N-sp}{p}\right|^p\,\Phi_{N,s,p}(r)\,dr>0,
\]
	 where, for every $0<r<1$, the quantity $\Phi_{N,s,p}(r)$ is given by
	 \begin{equation*}
	 	\Phi_{N,s,p}(r)=|\mathbb{S}^{N-2}|\,\int_{-1}^1 \frac{(1-t^2)^\frac{N-3}{2}}{(1-2\,t\,r+r^2)^\frac{N+sp}{2}}\,dt,\qquad \mbox{ for }N\ge 2,
	 \end{equation*}
	 and
	 \[
	 \Phi_{1,s,p}(r)=\frac{1}{(1-r)^{1+sp}}+\frac{1}{(1+r)^{1+sp}}.
	 \]
The case of convex sets has been considered recently in \cite{BBZ}, where, in Theorems 6.3 and 6.6,  it has been proved that, if $\Omega$ is convex, the optimal constant $\mathfrak{h}_{s,p}(\Omega)$ coincides with  the one of the half-space $\mathbb H^N_+:=\mathbb R^{N-1}\times (0,+\infty)$ (whose explicit value is given in  formulas (1.9)-(1.10) in \cite{BBZ})  in the following situations:
	\begin{itemize}
		\item for $1<p<\infty$ and  $1/p \le s <1$;
		\item for $p=2$ and $0<s<1$.
	\end{itemize}

We note that  the paper \cite{BBZ} extends some previous results contained in \cite{BD,FMT} for the case $p=2$. More precisely in \cite[Theorem 1.1]{BD}, the explicit value of the optimal constant for the half-space has been computed for $p=2$ and any $0<s<1$, while in \cite[Theorem 5]{FMT}, it has been proved that if $\Omega$ is convex, then $\mathfrak{h}_{s,2}(\Omega)=\mathfrak{h}_{s,2}(\mathbb H^N_+)$ for any $1/2\le s<1$.

When $sp>N$,  the recent paper \cite{S} shows  that the fractional Hardy inequality  \eqref{fractional-Hardy} holds on every open set  $\Omega\subset\mathbb R^N$,  by adapting the technique in \cite{H} to the non-local setting. However, such an approach does not give  any  lower estimate on the fractional Hardy constant $\mathfrak{h}_{s,p}(\Omega)$. 

With the aim to provide an optimal lower bound on $\mathfrak{h}_{s,p}(\Omega)$ when    $\Omega \subset \mathbb R^N$ is a general open set and $sp>N$,   in this paper we give a different proof of  the Hardy inequality  \eqref{fractional-Hardy} which comes out with a lower sharp  estimate on $\mathfrak{h}_{s,p}(\Omega)$. In particular,  our main result  extends  inequality \eqref{s=1}  to the fractional case  $sp>N$. 

\begin{theorem}\label{main} 
Let $N\ge 1$, $0<s<1$ and $1<p<\infty$ be such that $sp>N$. For every open set $\Omega\subsetneq\mathbb{R}^N$ we have
\begin{equation}
\label{mainstima}
\mathfrak{h}_{s,p}(\Omega)\geq \mathfrak{h}_{s,p},\qquad \text{where}\ \mathfrak{h}_{s,p}:=\mathfrak{h}_{s,p}(\mathbb{R}^N\setminus\{0\}).
\end{equation}
	 \end{theorem}

The proof of this result is based on the so-called \textit{supersolution method}, which was well known in the classical local case (see e.g. \cite{Ancona, KK, DP})
and was extensively studied in the fractional setting in \cite{BBSZ, BBZ}. Such a method allows to give an equivalent ``dual" definition of $\mathfrak h_{s,p}(\Omega)$, which relies on the existence of positive supersolutions to the nonlinear fractional equation:
\begin{equation}\label{EL-super}
	(-\Delta_p)^s u =\lambda \frac{|u|^{p-2}u}{d_\Omega^{sp}} \quad \mbox{in } \Omega,
	\end{equation}
where $(-\Delta_p)^s$ denotes the fractional $p$-Laplacian, whose precise definition will be given later on in Section \ref{sec:2}. We recall that, in \cite[Theorem 1.1]{BBSZ}, it is proved that 
\begin{equation}\label{dual-super}
	\mathfrak h_{s,p}(\Omega)=\sup\{\lambda \ge 0\,:\, \mbox{equation} \ (\ref{EL-super}) \,\, \hbox{admits a positive local weak supersolution} \}.
	\end{equation}
For the definition of local weak super/subsolution we refer to Definition \ref{def-sol} below.

Such result, as well explained in \cite{BBSZ}, is based on the equivalence between the strict positivity of $\mathfrak h_{s,p}(\Omega)$ and the existence of a positive (local weak) supersolution to \eqref{EL-super} for some $\lambda$. For more details on the supersolution method see \cite{BBSZ} and reference therein.

Thanks to the formula \eqref{dual-super} above, in order to prove Theorem \ref{main}, it is enough to find a positive supersolution to \eqref{EL-super} for $\lambda=\mathfrak h_{s,p}(\mathbb R^N\setminus \{0\})$. This is the content of Theorem \ref{theorem:super-beta} (and, more precisely, of Corollary \ref{theorem:main-super} below), where we give an explicit supersolution  to  \eqref{EL-super}, in terms of powers of the distance function. 
\medskip

In the second part of the paper we study the asymptotics, when $s\nearrow 1$  of $\mathfrak h_{s,p}$, as well as its limit when   $p \nearrow \infty$.
The strategy adopted  in \cite{S}  does not allow to   perform a quantitative study of the optimal Hardy constant $\mathfrak{h}_{s,p}(\Omega)$ and of  its behaviour  as  $s\nearrow 1$ and $p\nearrow \infty$.  On the contrary, the application of the supersolution method  permits us to prove,  that, when $p>N$, it holds
\[
\lim_{s \nearrow 1}(1-s)\mathfrak h_{s,p}(\mathbb R^N\setminus\{0\})  = K_{p,N} \left(\frac{p-N}{p}\right)^p=K_{p,N}  \mathfrak h_{p}(\mathbb R^N\setminus\{0\}),
\]
where $K_{p,N}$ is an explicit constant depending only on $p$ and $N$ (see Theorem \ref{s-to-1}). This will follow by combining a \textit{limsup}-inequality (which is valid for any open subset of $\mathbb R^N$) and a \textit{liminf}-inequality (proved  for $\Omega=\mathbb R^N\setminus \{0\}$), which are established in Lemma \ref{limsup}-\ref{liminf}, respectively. We emphasize that, while for proving the $\limsup$-inequality, it is sufficient to use the variational definition \eqref{inf-W} of $\mathfrak h_{s,p}(\Omega)$,  for establishing the liminf-inequality the dual formulation \eqref{dual-super} is better situated (since there, the Hardy constant is written as a \textit{supremum} rather than an \textit{infimum}).

Moreover, by exploiting the lower bound  \eqref{mainstima},  in  Theorem \ref{pinfty}  we show that, for every $0<s<1$, it holds
\begin{equation}
	\label{hardyinftyintro}
	\lim_{p\to \infty} (\mathfrak{h}_{s,p}(\Omega))^{\frac{1}{p}}=1,
\end{equation}
 generalising  the result given in \cite[Theorem 4.4]{BPZ} when $s=1$.
Again, this will follow by combining a \textit{liminf} and a \textit{limsup} inequality, both valid, now, for any open set $\Omega$.

In the last section of this paper,  we apply  Theorem \ref{main} to obtain the following  Cheeger inequality 
\begin{equation}\label{stimacheegerintro}\lambda_{s,p}(\Omega)\geq  \mathfrak{h}_{s,p} \left(\frac{h_1(\Omega)}{N}\right)^{sp},\end{equation}
for $sp>N$ and $\Omega\subsetneq \mathbb R^N$ open, see  Theorem \ref{cheeger}. 
Here  $h_{1}(\Omega)$ is the  {\it Cheeger constant}  of $\Omega$ defined by   \begin{equation}\label{cheegcon}
h_{1}(\Omega)=\inf\bigg\{\frac{P(E)}{|E|}\ :\ E\Subset\Omega\ \text{smooth},\, |E|>0\bigg\},
\end{equation}
and 
 $ \lambda_{s,p}(\Omega)$  is defined by the following sharp fractional Poincar\'e inequality 
\begin{equation}\label{speigenv}
\lambda_{s,p}(\Omega)=\inf_{u\in C^\infty_0(\Omega)}\left\{[u]^p_{W^{s,p}(\mathbb{R}^N)} :\, \int_{\Omega} |u|^p \, dx=1\right\}.
\end{equation}
  We explicitly note that,  in the case $s=1$,  when  $p>N$,   combining \eqref{stimacheegerintro}   with  \eqref{s=1},  we get  an improvement of the classical Cheeger inequality  with a constant which does not vanish as $p\to \infty$ (for further details, see Remark \ref{cheegerimpr}). Finally, thanks to \eqref{hardyinftyintro}, we   study the  asymptotic behaviour of the family $\left(\lambda_{s,p}(\Omega)\right)^{1/p}$ as $p\to \infty$ (see Corollary \ref{convergenzaautov}), getting a sharp estimate in the limit case $p=\infty$ (see  \eqref{cheegerbelow}).

\begin{ack} We wish to thank Lorenzo Brasco  for some useful discussions. 
The authors are members of the {\it Gruppo Nazionale per l'Analisi Matematica, la Probabilit\`a
e le loro Applicazioni} (GNAMPA) of the Istituto Nazionale di Alta Matematica (INdAM). E.C. is partially supported by the PRIN project 2022R537CS \emph{$NO^3$ - Nodal Optimization, NOnlinear elliptic equations, NOnlocal geometric problems, with a focus on regularity. } The work of F.P. has been  financially supported by the PRA-2022 project \emph{Geometric evolution problems and PDEs on variable domains}, University of Pisa.  
\end{ack}

\section{An explicit supersolution and the proof of Theorem \ref{main}}
\label{sec:2}
For every $1<p<\infty$, we indicate by $J_p:\mathbb{R}\to \mathbb{R}$ the monotone increasing function defined by 
\[J_p(t)=|t|^{p-2}\,t,\qquad \mbox{ for every } t\in\mathbb{R}.
\]
For $x_0\in\mathbb{R}^N$ and $R>0$, we will denote by $B_R(x_0)$ the $N-$dimensional open ball centered at $x_0$, with radius $R$. We will use the standard notation $\omega_N$ for the $N-$dimensional Lebesgue measure of $B_1(0)$.  For an open set $\Omega\subsetneq \mathbb{R}^N$, we denote by 

\[ 
d_{\Omega}(x):=\min_{y \in \partial\Omega} |x-y|, \qquad \mbox{ for every } x \in \Omega,
\]
the distance function from the boundary. We extend $d_{\Omega}$  by $0$ outside $\Omega$.  Moreover, we denote by $r_{\Omega}$  the inradius of $\Omega$, defined by 
\[
r_\Omega=\|d_\Omega\|_{L^\infty(\Omega)}=\sup \Big\{r>0\, :\, \mbox{there exists }x_0\in \Omega \mbox{ such that } B_r(x_0)\subseteq\Omega\Big\}.
\]

For a pair of open sets $E\subseteq \Omega\subseteq\mathbb{R}^N$, the symbol $E\Subset \Omega$ means that the closure $\overline{E}$ is a compact subset of $\Omega$.

For $0<\alpha<\infty$, we  denote by $L^{\alpha}_{sp}(\mathbb{R}^N)$ the following weighted Lebesgue space

\[L^{\alpha}_{sp}(\mathbb{R}^N)=\left\{  u\in L^{\alpha}_{\rm loc}(\mathbb{R}^N): 
\int_{\mathbb{R}^N} \frac{|u|^{\alpha}}{(1+|x|)^{N+sp}}\,dx<+\infty\right\}.\]

For $1<p<\infty$, $0<s<1$, and $\Omega\subsetneq \mathbb{R}^N$ open, we will  consider the equation 
\begin{equation}\label{eq:deltaps}(-\Delta_p)^s u=\lambda \frac{|u|^{p-2}u}{d_{\Omega}^{sp}}\quad \mbox{in } \Omega,\end{equation}
where $\lambda\geq 0$. 
Here $(-\Delta_p)^s$ is the {\it fractional $p-$Laplacian of order $s$}, defined in its weak form by 
\[
\langle(-\Delta_p)^s u,\varphi\rangle:=\iint_{\mathbb{R}^N\times \mathbb{R}^N} \frac{J_p(u(x)-u(y))\,(\varphi(x)-\varphi(y))}{|x-y|^{N+sp}}\,dx\,dy, \qquad \mbox{ for every } \varphi\in C^\infty_0(\mathbb{R}^N).
\]
\begin{definition}\label{def-sol} We say that $u\in W^{s,p}_{\rm{loc}}(\mathbb{R}^N)\cap L^{p-1}_{sp}(\mathbb{R}^N)$ is a 
\begin{itemize}
\item { \sl local weak supersolution} of \eqref{eq:deltaps} if
\begin{equation*}
\iint_{\mathbb{R}^N\times\mathbb{R}^N} \frac{J_p(u(x)-u(y))\,(\varphi(x)-\varphi(y))}{|x-y|^{N+sp}}\,dx\,dy\geq \lambda \int_{\mathbb{R}^N}  \frac{|u|^{p-2}u}{d_{\Omega}^{sp}}\varphi\,dx\end{equation*}
for every non-negative $ \varphi\in C^\infty_0(\Omega)$;
\vskip.2cm
\item { \sl local weak subsolution} of \eqref{eq:deltaps} if
\begin{equation*}
\iint_{\mathbb{R}^N\times\mathbb{R}^N} \frac{J_p(u(x)-u(y))\,(\varphi(x)-\varphi(y))}{|x-y|^{N+sp}}\,dx\,dy\leq \lambda \int_{\mathbb{R}^N}  \frac{|u|^{p-2}u}{d_{\Omega}^{sp}}\varphi\,dx\end{equation*}
for every non-negative $ \varphi \in  C^\infty_0(\Omega)$.
\item { \sl  local weak solution} of \eqref{eq:deltaps} if it is both a local weak supersolution and a  local weak subsolution.

\end{itemize}
\end{definition}

The aim of this section is to prove the following result:

\begin{theorem}\label{theorem:super-beta}
	Let $sp>N$ and	let $\Omega\subsetneq \mathbb R^N$ be an open set. 
	Then, the function
	$$U_{\Omega,\beta}:=d_{\Omega}^{\beta}, \qquad \beta \in \left(0,\frac{ps-N}{p-1}\right),$$
	is a positive local weak supersolution of
	\begin{equation}\label{super-beta}
		(-\Delta_p)^s U_{\Omega,\beta}=\mathcal{C}(\beta) \frac{U_{\Omega,\beta}^{p-1}}{d_{ \Omega}^{sp}},\qquad \mbox{in } \Omega,
	\end{equation}
where the positive constant $\mathcal{C}(\beta)$ is given by 
\begin{equation}\label{cbeta}
\mathcal{C}(\beta):=4\pi \alpha_N\int_0^1|1-\rho^\beta|^{p-2}(1-\rho^\beta)\left[\rho^{N-1}-\rho^{ps-\beta(p-1)-1}\right]G(\rho^2)\,d\rho ,
\end{equation}
with 
$$
\alpha_N:=\frac{\pi^{\frac{N-3}{2}}}{\Gamma\left(\frac{N-1}{2}\right)}, \qquad  G(t):=B\left(\frac{N-1}{2},\frac{1}{2}\right)F\left(\frac{N+ps}{2},\frac{ps+2}{2};\frac{N}{2};t\right).
$$
Here $\Gamma,\,B,$ and $F$ denote the gamma, the beta, and the hypergeometric functions, respectively\footnote{We refer for example to \cite[Chapter 1]{Leb} and \cite[Chapter 9]{Leb} for the definitions and properties of these functions.}.

\end{theorem}
\begin{remark}
\label{rem:costanti}
Observe that, for the choice $\beta=(sp-N)/ p$, we have that 
\[
\mathcal{C}\left(\frac{sp-N}{p}\right)=\mathfrak{h}_{s,p}.
\]
Indeed, with simple algebraic manipulations, this choice gives
\[
\begin{split}
\left|1-\rho^\frac{sp-N}{p}\right|^{p-2}\left(1-\rho^\frac{sp-N}{p}\right)\left[\rho^{N-1}-\rho^{ps-\frac{sp-N}{p}(p-1)-1}\right]&=\left(1-\rho^\frac{sp-N}{p}\right)^p\rho^{N-1}\\
&=\rho^{sp-1}\,\rho^{N-sp}\,\left(1-\rho^\frac{sp-N}{p}\right)^p\\
&=\rho^{sp-1}\,\left|1-\rho^\frac{N-sp}{p}\right|^p.
\end{split}
\]
Moreover, as observed in \cite{FS}, according to \cite[equation (3.665)]{GR} we have
\[
\Phi_{N,s,p}(t)=|\mathbb{S}|^{N-2}\,B\left(\frac{N-1}{2},\frac{1}{2}\right)F\left(\frac{N+ps}{2},\frac{ps+2}{2};\frac{N}{2};t\right)=|\mathbb{S}^{N-2}|G(t^2),
\]
and 
\[
|\mathbb{S}^{N-2}|=2\,\frac{\pi^{\frac{N-1}{2}}}{\Gamma\left(\frac{N-1}{2}\right)}=2\,\pi\,\alpha_N.
\]
This  finally gives
\[
\mathfrak{h}_{s,p}=2\,\int_0^1 \rho^{sp-1}\,\left|1-\rho^\frac{N-sp}{p}\right|^p\,\Phi_{N,s,p}(\rho)\,d\rho=\mathcal{C}\left(\frac{sp-N}{p}\right).
\]
\end{remark} 
According to the previous remark, when $\beta=(sp-N)/ p$ from Theorem \ref{theorem:super-beta} we immediately get the following
\begin{corollary}\label{theorem:main-super}
Let $sp>N$ and	let $\Omega\subsetneq \mathbb R^N $ be an open set. 
	Then, the function
	$$
	U_\Omega:=d_{\Omega}^{\frac{sp-N}{p}},
	$$
		is a positive local weak supersolution of
		\begin{equation*}
			(-\Delta_p)^s U_\Omega = \mathfrak{h}_{s,p}  \frac{U_\Omega^{p-1}}{d_{\Omega}^{sp}}\quad \mbox{in } \Omega.
			\end{equation*}
\end{corollary}

In order to show Theorem \ref{theorem:super-beta}, we recall that, by \cite[Theorem 1.1]{DPQ} (see also \cite[Lemma 3.1]{FS}), given $z\in \mathbb R^N$,  $sp > N$ and  $0<\beta<(sp-N)/(p-1)$,  the function
\[
	V(x):=|x-z|^{\beta}=d_{\mathbb R^N\setminus\{z\}}^{\beta}(x),
\]
belongs to $W^{s,p}_{\rm{loc}}(\mathbb{R}^N)\cap L^{p-1}_{sp}(\mathbb{R}^N)$ and is a local weak solution to
\begin{equation}\label{EL-eq}
	(-\Delta_p)^s V=
	\mathcal{C}(\beta)\frac{V^{p-1}}{d_{\mathbb R^N\setminus{\{z\}}}^{sp}}\quad \mbox{in }\mathbb R^N\setminus \{z\}.
\end{equation} 
where  $\mathcal{C}(\beta)$ is  given by \eqref{cbeta}.
Using this fact, we can prove the following preliminary result.

\begin{proposition}\label{propkey}
	Let $sp >N$ and $0<\beta<(sp-N)/(p-1)$. Let $x_0,\,x_1 \in \mathbb R^N$. Then, the function 
	$$U_1:=d_{\mathbb R^N\setminus\{x_0,x_1\}}^{\beta}$$
	is a local weak supersolution of 
	\begin{equation}\label{super1}
		(-\Delta_p)^s U_1=\mathcal{C}(\beta)\frac{U_1^{p-1}}{d_{\mathbb R^N\setminus{\{x_0,x_1\}}}^{sp}},\qquad \text{in}\ \mathbb R^N\setminus \{x_0,x_1\},
	\end{equation}
	where  $\mathcal{C}(\beta)$ is  given by \eqref{cbeta}.

\end{proposition}

\begin{proof}
	We start with the obvious (yet crucial) observation that, since $\beta>0$, we have 
	
	\begin{equation}\label{crucial}
	\begin{split}
		U_1=d_{\mathbb R^N\setminus\{x_0,x_1\}}^{\beta}&=\left(\min\{d_{\mathbb R^N\setminus\{x_0\}},d_{\mathbb R^N\setminus\{x_1\}}\}  \right)^{\beta}\\
		&=\min\left\{d_{\mathbb R^N\setminus\{x_0\}}^{\beta}, d_{\mathbb R^N\setminus\{x_1\}}^{\beta}\right\}=\min{\{V_0,V_1\}},
		\end{split}
	\end{equation}
	where
	$$V_i(x)=|x-x_i|^{\beta},\qquad \text{for}\ i=0,\,1.$$
Then the function  $U_1$ belongs to $W^{s,p}_{\rm{loc}}(\mathbb{R}^N)\cap L^{p-1}_{sp}(\mathbb{R}^N)$.
	In order to prove that $U_1$ is a local  weak supersolution of \eqref{super1},  we can suitably adapt the strategy applied in the proof of  \cite[Theorem 1.1]{KKP}, where it is shown that  the minimum of two locally weakly $(s,p)$-superharmonic functions is itself a  locally weakly $(s,p)$-superharmonic function. 
In our case we need to handle the additional nonlinear term on the right-hand side, which is however local. 
\par
	Let  $\varphi\in C_0^{\infty}(\mathbb R^N\setminus\{x_0,x_1\})$ be a nonnegative test function. Then it  is admissible for the weak formulation of the equations satisfied by $V_0$ and $V_1$ (i.e. equation \eqref{EL-eq} with $z$ replaced by $x_0$ and $x_1$, respectively). For every $0<\varepsilon<1/4$ we define  $$\theta_{\varepsilon}:=\min\left\{ 1, \frac{(V_0-V_1)_{+}}{\varepsilon}\right\}.$$
	Then we consider $\eta_{1}=(1-\theta_{\varepsilon}) \varphi$  as a test function 
in  the  equation satisfied by $V_0$, and $\eta_{2}=\theta_{\varepsilon} \varphi$  as a test function  in the equation satisfied by $V_1$.  By summing up the
corresponding integrals for $V_0$ and $V_1$, we obtain
\begin{equation}\label{scomp}
\mathcal{C}(\beta) \int_{\mathbb{R}^N}  \frac{|V_0|^{p-2}V_0}{d_{\mathbb R^N\setminus\{x_0\}}^{sp}} (1-\theta_{\varepsilon}) \varphi dx+  \mathcal{C}(\beta) \int_{\mathbb{R}^N}  \frac{|V_1|^{p-2}V_1}{d_{\mathbb R^N\setminus\{x_1\}}^{sp}}\theta_{\varepsilon} \varphi dx  =\iint_{\mathbb{R}^N\times\mathbb{R}^N} \frac{ \Phi_{\varepsilon}(x,y)}{|x-y|^{N+sp}}\,dx\,dy, 
\end{equation}
where 
\[ 
\begin{split} \Phi_{\varepsilon}(x,y)&=J_p(V_0(x)-V_0(y)) ((1-\theta_{\varepsilon}(x)) \varphi(x)- (1-\theta_{\varepsilon}(y)) \varphi(y))\\
&\hspace{1em}+ J_p(V_1(x)-V_1(y)) (\theta_{\varepsilon}(x) \varphi(x)- \theta_{\varepsilon}(y) \varphi(y)). 
\end{split} 
\]
As shown in  \cite[Theorem 1.1]{KKP}, we have that 
	\begin{equation} \label{superh}\ \limsup_{\varepsilon\to 0 }\iint_{\mathbb{R}^N\times\mathbb{R}^N} \frac{ \Phi_{\varepsilon}(x,y)}{|x-y|^{N+sp}}\,dx\,dy \leq  \iint_{\mathbb{R}^N\times\mathbb{R}^N} \frac{J_p(U_1(x)-U_1(y))\,(\varphi(x)-\varphi(y))}{|x-y|^{N+sp}}\,dx\,dy.\end{equation}
We  further observe that,  by \eqref{crucial}, we have 
\begin{equation}\label{r.h.s}
		\frac{U_1^{p-1}(x)}{d_{\mathbb R^N\setminus\{x_0,x_1\}}^{sp}(x)}=\frac{\min\{V_0^{p-1}(x),V_1^{p-1}(x)\}}{\min\left\{d_{\mathbb R^N\setminus\{x_0\}}^{sp}(x),d_{\mathbb R^N\setminus \{x_1\}}^{sp}(x)\right\}}=
		{\begin{cases}
				\displaystyle	\frac{V_0^{p-1}(x)}{d_{\mathbb R^N\setminus\{x_0\}}^{sp}(x)} & \mbox{if }\, x\in S_1,\\
				\\
				\displaystyle	\frac{V_1^{p-1}(x)}{d_{\mathbb R^N\setminus\{x_1\}}^{sp}(x)} & \mbox{if  }\,x\in S_2,
		\end{cases}}
	\end{equation}
where 
$$
S_1:=\{x\in \mathbb R^N:\, |x-x_0|\le |x-x_1| \} \ \hbox{ and  } \  S_2:=\{x\in \mathbb R^N:\, |x-x_1|< |x-x_0| \}.
$$
Moreover, by definition of $\theta_{\varepsilon}$,  $S_1$ and $S_2$, we have that 
 $$\theta_{\varepsilon}=0 \  \hbox{on } \ S_1 \quad \hbox{ and } \  \theta_{\varepsilon} \to 1 \hbox{ pointwise on }  S_2 .$$
Then, taking into account that $\varphi\in C_0^{\infty}(\mathbb R^N\setminus\{x_0,x_1\})$, we can  apply the Lebesgue Dominated Convergence Theorem and we  get 
	\begin{equation} \label{nonline}\begin{split}
&\lim_{\varepsilon\to 0}   \int_{\mathbb{R}^N}  \frac{V_0^{p-1}}{d_{\mathbb R^N\setminus\{x_0\}}^{sp}} (1-\theta_{\varepsilon}) \varphi dx+  \int_{\mathbb{R}^N}  \frac{V_1^{p-1}}{d_{\mathbb R^N\setminus\{x_1\}}^{sp}}\theta_{\varepsilon} \varphi dx \\
&=    \int_{S_1}  \frac{V_0^{p-1}}{d_{\mathbb R^N\setminus\{x_0\}}^{sp}}  \varphi dx + \lim_{\varepsilon\to 0}  \left(\int_{S_2}  \frac{V_0^{p-1}}{d_{\mathbb R^N\setminus\{x_0\}}^{sp}}(1-\theta_{\varepsilon}) \varphi dx+ \int_{S_2}  \frac{V_1^{p-1}}{d_{\mathbb R^N\setminus\{x_1\}}^{sp}}\theta_{\varepsilon} \varphi dx\right)\\ 
&= \int_{S_1}  \frac{V_0^{p-1}}{d_{\mathbb R^N\setminus\{x_0\}}^{sp}}\varphi dx+	\int_{S_2}  \frac{V_1^{p-1}}{d_{\mathbb R^N\setminus\{x_1\}}^{sp}}\varphi dx=	\int_{\mathbb{R}^N}  \frac{U_1^{p-1}}{d_{\mathbb R^N\setminus\{x_1,x_0\}}^{sp}}\varphi dx,\
\end{split}
	\end{equation}
	where the last identity follows   thanks to  \eqref{r.h.s}. 
Hence, combining \eqref{superh}, \eqref{scomp} and \eqref{nonline}, we get the desired conclusion.
\end{proof}

By repeatedly applying Proposition \ref{propkey},  we have the following corollary.
\begin{corollary}\label{cor-super}
Let $sp >N$ and $0<\beta<(sp-N)/(p-1)$.	Let $n\in \mathbb N$ and $x_0,\dots,x_n \in \mathbb R^N$. Then, the function
	$$
	U_n=d_{\mathbb R^N\setminus\{x_0,\dots,x_n\}}^{\beta},	
	$$
	is a local weak supersolution of 
	\begin{equation}\label{super-n}
		(-\Delta_p)^s U_n= \mathcal{C}(\beta)\frac{U_n^{p-1}}{d_{\mathbb R^N\setminus{\{x_0,\dots,x_n\}}}^{sp}},\qquad \text{in}\ \mathbb R^N\setminus \{x_0,\dots,x_n\},
	\end{equation}
	where $\mathcal{C}(\beta)$ is given by \eqref{cbeta}.

\end{corollary}

We are now ready to prove Theorem \ref{theorem:super-beta}. The idea is to argue by approximation: we pick a countable dense subset $\cup_{i\in \mathbb N}\{x_i\}$ of $\partial\Omega$, we apply  Corollary \ref{cor-super} on the finite set $\cup_{i=0}^n\{x_i\}$,  and finally we pass to the limit as $n\to \infty$.

\begin{proof}[Proof of Theorem \ref{theorem:super-beta}]
	Let $D=\cup_{i\in \mathbb N} \{x_i\}$ be a dense subset of $\partial \Omega$ and  for every $ n\in \mathbb{N} $ we define 
	$$
	E_n=\mathbb R^N \setminus\{x_0,\dots,x_n\} \qquad \text{and}\qquad U_n=d^{\beta}_{E_n}.
	$$  
	By Corollary \ref{cor-super}, we know that, for every $n\in \mathbb N$, the function $U_n\in L_{sp}^{p-1}(\mathbb R^{N})$ is a local weak supersolution of \eqref{super-n} in $ E_n\supseteq \Omega$. Moreover, $\{d_{E_n}  \}_{n\in\mathbb{N}}$ and  $\{U_n\}_{n\in\mathbb{N}}$ are  decreasing sequences  (being $\beta>0$) and, since $D$ is dense in  $\partial \Omega$, we have that for every $x\in\Omega$
	 \begin{equation}
	 \label{limdist}
	 d_{ \Omega}(x)= \min_{z\in \partial \Omega} |x-z|= \inf_{i\in \mathbb{N}}  |x-x_i|= \inf_{n\in \mathbb{N}} d_{E_n}(x)= \lim_{n\to \infty} d_{E_n}(x),  
	 \end{equation}	
	 and 
	 \begin{equation}
	 \label{limitU} 
	 U_{\Omega,\beta}(x)=d^{\beta}_{ \Omega}(x)= \min_{z\in \partial \Omega} |x-z|^{\beta}= \inf_{i\in \mathbb{N}}  |x-x_i|^{\beta} =   \inf_{n\in \mathbb{N} } U_n(x)= \lim_{n\to \infty} U_n (x).\end{equation}
	In particular, $U_{\Omega,\beta}\in L_{sp}^{p-1}(\mathbb R^{N})$.
	In order to show that the function $U_{\Omega,\beta}$ is a local weak supersolution of \eqref{super-beta},  let $\varphi \in C^\infty_0(\Omega)$ be a non-negative function and we observe that it is an admissible test function for the weak formulation of the equation satisfied by $U_n$, for any $n \in \mathbb N$. Indeed, we have $\Omega\subseteq E_n$ by costruction.
Hence, for any $n\in \mathbb N$, we have
\begin{equation}\label{super2}
\iint_{\mathbb R^N \times \mathbb R^N} \frac {J_p(U_n(x)-U_n(y))(\varphi(x)-\varphi(y))}{|x-y|^{N+sp}}\,dx\,dy \ge \mathcal{C}(\beta) \int_{\Omega}\frac{U_n^{p-1}\varphi}{d^{sp}_{E_n}}\,dx.
\end{equation}
Since the integrand in the right-hand side of \eqref{super2} is non-negative, Fatou's Lemma in conjunction with \eqref{limdist} and  \eqref{limitU} gives
\begin{equation}
\label{r.h.s.}
	\liminf_{n\to \infty}\int_{\mathbb R^N}\frac{U_n^{p-1}\varphi}{d^{sp}_{E_n}}\,dx\ge \int_{\mathbb R^N}\frac{U_{\Omega,\beta}^{p-1}\varphi}{d^{sp}_{\Omega}}\,dx.	
\end{equation}
Let us consider now the left-hand side of \eqref{super2}.  We recall that $\varphi$ is compactly supported in $\Omega$, thus,  denoting by  $S_\varphi$ its support, we have that 
\[
\delta:=\mathrm{dist}(\partial \Omega, S_\varphi) >0.
\]
Moreover, we set
$$
K_{\delta}(\varphi)= \{y \in \Omega\,:\, \mathrm{dist}(y,S_\varphi)<  \delta/2\}.
$$ 
We can write

\[\begin{split}
	&\iint_{\mathbb R^N\times \mathbb R^N} \frac{J_p(U_n(x)-U_n(y))(\varphi(x)-\varphi(y))}{|x-y|^{N+sp}}\,dx\,dy\\
	&\hspace{1em} = \iint_{S_\varphi\times S_\varphi} \frac{J_p(U_n(x)-U_n(y))(\varphi(x)-\varphi(y))}{|x-y|^{N+sp}}\,dx\,dy \\
	&\hspace{2em}
	+2\iint_{S_\varphi\times (\mathbb R^N\setminus S_\varphi)} \frac{J_p(U_n(x)-U_n(y))(\varphi(x)-\varphi(y))}{|x-y|^{N+sp}}\,dx\,dy\\
	&\hspace{1em} = \iint_{S_\varphi\times S_\varphi}   \frac{J_p(U_n(x)-U_n(y))(\varphi(x)-\varphi(y))}{|x-y|^{N+sp}}\,dx\,dy \\
	&\hspace{2em}+	2\iint_{S_\varphi \times (K_{\delta}(\varphi)\setminus S_\varphi)}\frac{J_p(U_n(x)-U_n(y))(\varphi(x)-\varphi(y))}{|x-y|^{N+sp}}\,dx\,dy \\
	& \hspace{2em} + 2  \iint_{S_\varphi \times  (\mathbb R^N\setminus K_{\delta}(\varphi))} \frac{J_p(U_n(x)-U_n(y))(\varphi(x)-\varphi(y))}{|x-y|^{N+sp}}\,dx\,dy=:\mathcal{I}^1_n+\mathcal{I}^2_n + \mathcal I^3_n.
	\end{split}\]

Taking into account the pointwise convergence  \eqref{limitU}, we want to apply the Dominated Convergence Theorem in order to pass to the limit in each $\mathcal I^i_n$.
To this aim, we observe that for $x,y\in\Omega$ we have
\[
|U_n(x)-U_n(y)|\le  \beta (d_{E_n} (x)^{\beta -1} + d_{E_n} (y)^{\beta -1})|d_{E_n} (x) - d_{E_n} (y)|.
\]
This simply follows from  the Fundamental Theorem of Calculus, applied to the function $t\mapsto t^{\beta}$. 
Moreover, by using that $d_\Omega\le d_{E_n}$, that $\beta-1<0$,   and the $1-$Lipschitz character of the distance function, we get
\[
|U_n(x)-U_n(y)|\le  \beta (d_{\Omega} (x)^{\beta -1} + d_{\Omega} (y)^{\beta -1})|x-y|,\qquad \text{for}\ x,y\in\Omega,
\]
and, in particular, we have that
\begin{equation}\label{Lip}
	|U_n(x)-U_n(y)| \le C(\beta,\delta)|x-y|,\quad \mbox{for any  } x,\, y\,\in K_{\delta}(\varphi) .
\end{equation}

Using \eqref{Lip} and the Lipschitz continuity of $\varphi$, we deduce

\begin{equation*}
		\frac{|J_p(U_n(x)-U_n(y))||\varphi(x)-\varphi(y)|}{|x-y|^{N+sp}}
	\le   C(\beta,\delta)^{(p-1)} ||\nabla \varphi||_{L^\infty(\Omega)}\ \frac{1}{|x-y|^{N+p(s-1)}} \in L^1(K_\delta(\varphi)\times K_\delta(\varphi)).
\end{equation*}

This allows to pass to the limit in both $\mathcal I^1_n$ and $\mathcal I^2_n$, showing that
\begin{equation}\label{I_1}\begin{split}
	\lim_{n\rightarrow +\infty} (\mathcal I_n^1+\mathcal I_n^2)&= \iint_{S_\varphi\times S_\varphi}   \frac{J_p(U_{\Omega,\beta}(x)-U_{\Omega,\beta}(y))(\varphi(x)-\varphi(y))}{|x-y|^{N+sp}}\,dx\,dy \\
&\hspace{1em}+	2\iint_{S_\varphi \times (K_{\delta}(\varphi)\setminus S_\varphi)}\frac{J_p(U_{\Omega,\beta}(x)-U_{\Omega,\beta}(y))(\varphi(x)-\varphi(y))}{|x-y|^{N+sp}}\,dx\,dy.
\end{split}
\end{equation}

It remains to show that

\begin{equation}\label{I_2+I_3}
		\lim_{n\rightarrow +\infty} \mathcal I_n^3=2\iint_{S_\varphi \times (\mathbb R^N\setminus K_{\delta}(\varphi))} \frac{J_p(U_{\Omega,\beta}(x)-U_{\Omega,\beta}(y))(\varphi(x)-\varphi(y))}{|x-y|^{N+sp}}\,dx\,dy.
		\end{equation}

We start by observing that, since  $\{U_n\}_{n\in\mathbb{N}}$ is a decreasing sequence, we have
\begin{equation*}\begin{split}
	&\frac{2|J_p(U_n(x)-U_n(y)||\varphi(x)-\varphi(y)|}{|x-y|^{N+sp}}\\
		&\hspace{3em}\le 4 \|\varphi\|_{L^\infty(\Omega)}\frac{|U_n(x)-U_n(y)|^{p-1}}{|x-y|^{N+sp}}\\
		&\hspace{3em}\le {2^{p} }\|\varphi\|_{L^\infty(\Omega)}\frac{|U_0(x)|^{p-1}+|U_0(y)|^{p-1}}{{|x-y|^{N+sp}}} .
			\end{split}
	\end{equation*}

Equality \eqref{I_2+I_3} will follow, by applying again the Dominated Convergence Theorem, if we show that 
\begin{equation}\label{g_L1}g(x,y):=\frac{|U_0(x)|^{p-1}+|U_0(y)|^{p-1}}{{|x-y|^{N+sp}}} \in L^1(S_\varphi \times (\mathbb R^N\setminus K_\delta(\varphi))).
\end{equation}

In order to do that,  we note that  
\begin{equation}\label{controllo} |x-y|>\frac \delta 2 \qquad \hbox{ for every } x\in  S_{\varphi} \hbox{ and }   y\in \mathbb R^N\setminus K_\delta(\varphi).\end{equation} Hence we have that

\begin{equation}\label{g-split}
	\begin{split}
&	\iint_{S_\varphi \times (\mathbb R^N\setminus K_\delta(\varphi))}g(x,y)\,dx\,dy=\int_{S_\varphi}|U_0(x)|^{p-1}\,dx \int_{\mathbb R^N\setminus K_\delta(\varphi)}\frac{1}{|x-y|^{N+sp}}\,dy \\
	&\hspace{14em} +	\iint_{S_\varphi\times (\mathbb R^N\setminus K_\delta(\varphi))} \frac{|U_0(y)|^{p-1}}{{|x-y|^{N+sp}}}\,dx\,dy\\
		&\leq \frac{N\omega_N}{sp} \left(\frac{2}{\delta}\right)^{sp}  \int_{S_\varphi}|U_0(x)|^{p-1}\,dx
 +\iint_{S_\varphi\times (\mathbb R^N\setminus K_\delta(\varphi))}\frac{|U_0(y)|^{p-1}}{{|x-y|^{N+sp}}}\,dx\,dy.
		\end{split}
\end{equation}

Since $S_{\varphi}$ is a compact set,   the first integral on the right-hand side in \eqref{g-split} is  finite. For the second integral, we observe that
for every   $x\in S_\varphi$ and for every $y\in \mathbb R^N\setminus K_\delta(\varphi)$, by applying again \eqref{controllo},    it holds 

\[\frac{ |y|+1}{ |x-y|}\leq \frac{ |y-x|+|x|+1}{ |x-y|} \leq  1+\frac{|x|+1}{ |x-y|}\leq 1+ \frac{2(M+1)}{\delta} =C,  \]
where $M=\max_{x\in S_\varphi} |x|$. This implies that
\begin{equation}\label{fuori}
	\frac{1}{|x-y|^{N+sp}}\leq  \frac{C^{N+sp}}{(1+|y|)^{N+sp}} \qquad \hbox{for every  $x\in S_\varphi$ and $y\in \mathbb R^N\setminus K_\delta(\varphi) $}.
\end{equation}

Since $U_0$ belongs to $L^{p-1}_{sp}(\mathbb R^N)$, the above estimate implies that also the second term on the right-hand side of \eqref{g-split} is finite and thus \eqref{g_L1} holds true.
\par
Finally, combining \eqref{I_1} and  \eqref{I_2+I_3}, we deduce that 
\[ \begin{split}
	&\lim_{n\to \infty}\iint_{\mathbb R^N \times \mathbb R^N} \frac {J_p(U_n(x)-U_n(y))(\varphi(x)-\varphi(y))}{|x-y|^{N+sp}}\,dx\,dy \\
	&\hspace{1em} = \iint_{\mathbb R^N \times \mathbb R^N} \frac {J_p(U_{\Omega,\beta}(x)-U_{\Omega,\beta}(y))(\varphi(x)-\varphi(y))}{|x-y|^{N+sp}}\,dx\,dy.
\end{split}
\]

Putting together the limit above with \eqref{super2} and \eqref{r.h.s.}, we conclude that $U_{\Omega,\beta}$ is a positive local weak supersolution of \eqref{super-beta}. 
\end{proof}

\par

Now we are in a position to prove  Theorem \ref{main}.
	\begin{proof}[Proof of Theorem \ref{main}]
	%
	By combining  formula  \eqref{dual-super}  with Theorem \ref{theorem:super-beta}  when $\beta=(sp-N)/ p$, we obtain that 
	\[
	\mathfrak{h}_{s,p}(\Omega) \geq  \mathcal{C}\left(\frac{sp-N}{p}\right)=\mathfrak{h}_{s,p}.
	\] 
The last equality follows from Remark \ref{rem:costanti}.	
	\end{proof}
We conclude this section with the following result. It gives a sufficient condition for the constant $\mathfrak{h}_{s,p}(\Omega)$ not to be attained. 
\begin{proposition}
	Let $1<p<\infty$, $0<s<1$  and let $\Omega\subsetneq\mathbb{R}^N$ be an open set. Suppose that there exists a positive local weak supersolution $u$ of \eqref{eq:deltaps} with $\lambda=\mathfrak{h}_{s,p}(\Omega)$, such that
	$$u \ge \frac{1}{C}\, d_\Omega^{\frac{sp-N}{p}},$$
	for some positive constant $C$.
	Then, the infimum $\mathfrak{h}_{s,p}(\Omega)$ is not attained.  
	\par
	
	In particular, when  $sp>N$, the constant $\mathfrak{h}_{s,p}(\Omega)$ is not attained for every open set $\Omega\subsetneq\mathbb{R}^N$ such that
	\[
	\mathfrak{h}_{s,p}(\Omega)=\mathfrak{h}_{s,p}.
	\]
	\end{proposition}
\begin{proof}[Proof]
The proof follows the one of \cite[Proposition 3.5]{BBZ} with some minor changes, and uses some integrability properties of the distance function. 
	We argue by contradiction and suppose that $v\in \widetilde{W}_0^{s,p}(\Omega)$ is a minimizer for $\mathfrak{h}_{s,p}(\Omega)$. Thus, in such a case, we have  $\mathfrak{h}_{s,p}(\Omega)>0$. Following \cite[Proposition 3.5]{BBZ}, we can assume that $v$ is positive. Let us consider now a sequence of functions $v_n \in C^{\infty}_0(\Omega)$ (which we can assume to be nonnegative) approximating $v$ in $W^{s,p}(\mathbb R^N)$ and almost everywhere.
\par	
By using as a  test function in the weak formulation of the inequality satisfied by $u$, the function
	$$
	\varphi:=\frac{v_n^p}{u^{p-1}},
	$$
	and proceeding as in \cite{BBZ}, the equality cases of the fractional Picone inequality permits to infer that
	$$
	u=Cv,\qquad \text{a.\,e. in}\ \Omega,
	$$ 
	for some positive constant $C$. Thus, that there exists another, possibly different, positive constant $C$ such that
	$$
	v \ge \frac{1}{C} d_\Omega^{\frac{sp-N}{p}},\qquad \text{in}\ \Omega.
	$$
	This contradicts the minimality of $v$, since we would have
	$$ [v]_{W^{s,p}(\mathbb R^N)}^p = \mathfrak h_{s,p}(\Omega)\int_\Omega \frac{v^p}{d_\Omega^{sp}}\,dx \ge \frac{\mathfrak h_{s,p}(\Omega)}{C^p}\int_\Omega \frac{1}{d_\Omega^N}\,dx = + \infty,$$
	where the last equality follows from \cite[Lemma 3.4]{BBSZ}.
	\par
	Finally, if $sp>N$, let us suppose that $\mathfrak h_{s,p}(\Omega)=\mathfrak{h}_{s,p}$. Then, thanks to Corollary \ref{theorem:main-super}, the function $U=d_\Omega^{(sp-N)/p}$ is a positive local weak supersolution of \eqref{eq:deltaps} with $\lambda=\mathfrak{h}_{s,p}=\mathfrak{h}_{s,p}(\Omega)$. From the first part of the proof we get the desired conclusion.  
	\end{proof}
	
	\section{Asymptotics for the sharp constant}

\subsection{The case $s\nearrow 1$}
This subsection is devoted to study the limit
\[
\lim_{s\nearrow 1}(1-s)\, \mathfrak{h}_{s,p}(\Omega),
\]
for $p>N$.
In the case when $\Omega$ is a convex set or the punctured space $\mathbb R^N\setminus\{0\}$ we can show that
$$
\lim_{s\nearrow 1 }(1-s)\, \mathfrak{h}_{s,p}(\Omega)=   K_{p,N} \mathfrak{h}_{p}(\Omega),
$$
where 
\[
\mathfrak{h}_{p}(\Omega)=\inf_{u\in C^\infty_0(\Omega)}\left\{\|\nabla u\|^p_{L^{p}(\Omega)} :\, \int_{\Omega} \frac{|u|^p}{d_{\Omega}^{p}} \, dx=1\right\}\quad \text{and}\quad K_{p,N}=\frac{1}{p}\,\int_{\mathbb{S}^{N-1}} |\langle \omega,\mathbf{e}_1\rangle|^p\,d\mathcal{H}^{N-1}(\omega).
\]
We start by recalling the celebrated {\it Bourgain-Brezis-Mironescu formula} 
	\begin{equation}\label{BBM}
		\lim_{s\nearrow 1} (1-s)\,[\varphi]^p_{W^{s,p}(\mathbb{R}^N)}=K_{p,N}\,[\varphi]^p_{W^{1,p}(\mathbb{R}^N)},\qquad \mbox{ for every } \varphi\in C^\infty_0(\mathbb{R}^N),
	\end{equation}
which will be useful in the sequel. For a proof of \eqref{BBM},  see for example \cite[Corollary 3.20]{EE}.
\par
Now we are in a position to state the main result of this section.   
\begin{theorem}\label{s-to-1}  Let  $p>N$ and let  $\Omega\subsetneq \mathbb R^N$ be an open set such that $\mathfrak{h}_{p}(\Omega)=\mathfrak{h}_{p}$.  Then 
\[\lim_{s\nearrow 1 }(1-s)\, \mathfrak{h}_{s,p}(\Omega)= {K_{p,N}} \mathfrak{h}_{p}={K_{p,N}}\bigg(\frac{p-N}{p}\bigg)^p.
\]
\end{theorem}
\par
In order to show Theorem \ref{s-to-1},  we start  proving  the $\limsup$ inequality, which holds true for any open set and every $p$.
\begin{lemma}\label{limsup}
Let $1<p<\infty$, for every open set $\Omega\subsetneq \mathbb R^N$, it holds
\begin{equation}
\label{limsupsto1} \limsup_{s\nearrow 1 }(1-s)\, \mathfrak{h}_{s,p}(\Omega) \leq  K_{p,N}\mathfrak{h}_{p}(\Omega).
\end{equation}

\end{lemma}
\begin{proof} For every $\varepsilon>0$,    there exists   $u_{\varepsilon}\in C_0^{\infty}(\Omega)$ such that 
\[
\mathfrak{h}_{p}(\Omega)+\varepsilon\geq \int_{\Omega} |\nabla u_{\varepsilon}|^p dx \qquad \mbox{ and } \qquad  \left\|\frac{u_{\varepsilon}}{d_{\Omega}}\right\|_{L^p(\Omega)}=1. 
\]
Then, by using Fatou's lemma and \eqref{BBM}, we get 
\[
\begin{split}
\limsup_{s\nearrow 1 }(1-s)\, \mathfrak{h}_{s,p}(\Omega)&\leq \limsup_{s\nearrow 1 }\frac{(1-s)\,[u_{\varepsilon}]^p _{W^{s,p}(\mathbb{R}^N)}   }{  \left\|\frac{u_{\varepsilon}}{d^s_{\Omega}}\right\|_{L^p(\Omega)}}\\
&\leq K_{p,N}\,\|\nabla u_{\varepsilon}\|^p_{L^p(\mathbb R^N)}\leq K_{p,N}\,\Big(\mathfrak{h}_{p}(\Omega)+\varepsilon\Big).
\end{split}
\]
By arbitrariness of $\varepsilon>0$, the latter gives \eqref{limsupsto1}.
\end{proof}

In order to prove the $\liminf$ inequality, we are going to show that for a  Sobolev function $u$, as $s\nearrow 1$, we have
\[
(1-s)\,(-\Delta_p)^s u \to -\Delta_p u, 
\]
in weak sense, up to a normalization constant. In other words, for every test function $\varphi$ we consider the limit, as $s$ goes to $1$, of
\[
(1-s)\,\langle(-\Delta_p)^su,\varphi\rangle:=(1-s)\iint_{\mathbb R^N\times \mathbb R^N} \frac{J_p(u(x)-u(y))\,(\varphi(x)-\varphi(y))}{|x-y|^{N+sp}}dx\,dy,
\]
and show that this coincides with
\[
K_{p,N}\int_{\Omega} |\nabla u|^{p-2}\nabla u\cdot \nabla \varphi\,dx.
\]
This extends \cite[Theorem 2.8]{BS} and \cite[Lemma 5.1]{dTGCV}, by considerably relaxing the assumptions on the involved functions.
The proof will exploit the convexity of the function $t\mapsto |t|^p$ and \eqref{BBM}.
 \begin{lemma}\label{alternative}
	Let $\Omega \subseteq \mathbb R^N$ be an open set and let $\varphi\in C^\infty_0(\Omega)$. Let $0<s<1$, $1<p<\infty$, and assume that $u \in W^{1,p}_{\rm loc}(\Omega)\cap L^{p-1}_{sp}(\mathbb R^N)$. Then,
	$$\lim_{s\nearrow 1}(1-s)\iint_{\mathbb R^N\times \mathbb R^N} \frac{J_p(u(x)-u(y))\,(\varphi(x)-\varphi(y))}{|x-y|^{N+sp}}dx\,dy = K_{p,N}\int_{\Omega} |\nabla u|^{p-2}\nabla u\cdot \nabla \varphi\,dx.$$
	
\end{lemma}

\begin{proof}
	Let  us denote by $S_\varphi$ the support of $\varphi$ and let $\Omega '\Subset\Omega$ be an open set with Lipschitz  boundary such that $S_\varphi \subseteq \Omega'$. By convexity of the map $t \rightarrow J_p(t)$, for every $t\in (0,1)$ we have:
	$$
	\frac{1}{p}[u+t\varphi]^p_{W^{s,p}(\Omega')}- \frac{1}{p}[u]^p_{W^{s,p}(\Omega ')} \ge t\, \iint_{\Omega ' \times \Omega'} \frac{J_p(u(x)-u(y))(\varphi(x)-\varphi(y))}{|x-y|^{N+sp}}\,dx\,dy.
	$$ 
Multiplying the above inequality by $(1-s)$, letting $s \nearrow 1$, and   applying  \cite[Corollary 1]{Ponce}, we deduce 
		\[ 
		\begin{split}
		\frac{K_{p,N}}{p}\left(\|\nabla u+t\nabla \varphi\|^p_{L^{p}(\Omega')}\right.&\left.- \|\nabla u  \|^p_{L^{p}(\Omega ')}\right) \\
		&\ge t\, \limsup_{s\nearrow 1}(1-s) \iint_{\Omega ' \times \Omega'} \frac{J_p(u(x)-u(y))(\varphi(x)-\varphi(y))}{|x-y|^{N+sp}}\,dx\,dy.
		\end{split}
		\]
Dividing by $t\in(0,1)$ and letting $t\searrow 0$, we have
	\begin{equation}
	\label{limitet}
	\begin{split}
	K_{p,N}&\int_{\Omega}|\nabla u|^{p-2} \nabla u\cdot \nabla \varphi\,dx\\
	& \ge \limsup_{s\nearrow 1} (1-s)\iint_{\Omega' \times \Omega'}\frac{J_p(u(x)-u(y))(\varphi(x)-\varphi(y))}{|x-y|^{N+sp}}\,dx\,dy.
	\end{split}\end{equation} 
We define \[ 
	T_s:=2\int_{\Omega'}\int_{\mathbb R^{N}\setminus \Omega'}\frac{J_p(u(x)-u(y))(\varphi(x)-\varphi(y))}{|x-y|^{N+sp}}\,dx\,dy=2\int_{ \Omega' }\varphi (x)\left(\int_{\mathbb R^N\setminus \Omega'}\frac{J_p(u(x)-u(y))}{|x-y|^{N+sp}}dy\right)\,dx.\]
We claim that $T_s$  is uniformly bounded for $s_0<s<1$, with a bound degenerating as $s_0$ goes to $0$. Indeed,  we note that 
\[
\delta:=\mathrm{dist}(\partial \Omega', S_\varphi)>0.
\] 
Then, we have
\[
\begin{split}
|T_s|&\leq 2\int_{S_\varphi} |\varphi(x)|\left(\int_{\mathbb R^N\setminus \Omega'  }\frac{|J_p(u(x)-u(y))|}{|x-y|^{N+sp}} dy\right)\,dx\\
	 	&\leq  2\|\varphi\|_{L^\infty(\Omega)}\ \iint_{S_\varphi\times\{y \in \mathbb R^N\,:\,d(y,S_\varphi)>\delta/2\}}\frac{|u(x)-u(y))|^{p-1}}{|x-y|^{N+sp}}\,dx\,dy\\
		&\le C_p \|\varphi\|_{L^\infty(\Omega)}\ \iint_{S_\varphi\times\{y \in \mathbb R^N\,:\,d(y,S_\varphi)>\delta/2\}}\frac{|u(x)|^{p-1}+|u(y)|^{p-1}}{{|x-y|^{N+sp}}}\,dx\,dy \\
	& \leq C_p\|\varphi\|_{L^\infty(\Omega)}\ \left(\frac{N\omega_N}{sp} \left(\frac{2}{\delta}\right)^{sp}  \int_{S_\varphi}|u(x)|^{p-1}\,dx+
	\iint_{S_\varphi\times\{y \in \mathbb R^N\,:\,d(y,S_\varphi)>\delta/2\}}\frac{|u(y)|^{p-1}}{{|x-y|^{N+sp}}}\,dx\,dy\right).\\
\end{split}
\]
By applying \eqref{fuori}, for every   $x\in S_\varphi$ we have that 
\[
\frac{|u(y)|^{p-1}}{|x-y|^{N+sp}}\leq  C^{N+sp} \frac{|u(y)|^{p-1}}{(1+|y|)^{N+sp}},\qquad \text{for every}\ y\in \mathbb R^N\ \text{such that}\  d(y,S_\varphi)>\delta/2.
\]
Since $u$ belongs to $L^{p-1}_{sp}(\mathbb R^N)$, the above estimate easily implies that 
$\{T_s\}_{s_0<s<1}$ is bounded.
Thus, we get 
\[
\lim_{s\nearrow 1}(1-s)T_s=0,
\]
and  \eqref{limitet}  gives 
	\[K_{p,N}\int_{\Omega}|\nabla u|^{p-2} \nabla u\cdot \nabla \varphi\,dx \ge \limsup_{s\nearrow 1} (1-s) \iint_{\mathbb R^N \times \mathbb R^N}\frac{J_p(u(x)-u(y))(\varphi(x)-\varphi(y))}{|x-y|^{N+sp}}\,dx\,dy.\]
	Finally, by replacing $\varphi$ with $-\varphi$, we get
	\[K_{p,N}\int_{\Omega}|\nabla u|^{p-2} \nabla u\cdot \nabla \varphi\,dx \le \liminf_{s\nearrow 1} (1-s) \iint_{\mathbb R^N \times \mathbb R^N}\frac{J_p(u(x)-u(y))(\varphi(x)-\varphi(y))}{|x-y|^{N+sp}}\,dx\,dy.
	\]
	Joining the last two equations,	we eventually conclude the proof.
			\end{proof}
		
		We can now give the proof of the $\liminf$ inequality, which holds true for the specific case of the punctured space $\Omega=\mathbb R^N\setminus\{0\}$.
		
		\begin{lemma}\label{liminf} Let $p>N$.   Then
$$\liminf_{s\nearrow1}(1-s) \mathfrak{h}_{s,p}\geq  K_{p,N}  \mathfrak{h}_{p}.$$
\end{lemma}
\begin{proof} Observe that 
\[
\beta:=\frac{p-N}p \in \left(0,\frac{sp-N}{p-1}\right),
\] 
for $s$ sufficiently close to $1$.
Hence, for such values of $s$,  by applying again \cite[Theorem 1.1]{DPQ}, the function $u(x)=  |x|^\beta$ is a positive local weak solution to
\[
(-\Delta_p)^s u=
	\mathcal{C}_{p,s}\frac{u^{p-1}}{|x|^{sp}},\quad \mbox{in }\mathbb R^N\setminus \{0\},
\]
where $\mathcal{C}_{p,s}$ is the constant given by \eqref{cbeta} when $\beta=(p-N)/p$.

By using again \eqref{dual-super}, we obtain that
\begin{equation}
\label{louerbaund}
 \mathfrak{h}_{s,p}\geq  \mathcal{C}_{p,s},
\end{equation}
for $s$ sufficiently close to $1$.
Moreover, by a direct computation, we have that  $u$ satisfies  
\[
-\Delta_p u= \mathfrak{h}_p \frac{u^{p-1}}{|x|^p},\quad \mbox{in }\mathbb R^N\setminus \{0\}.
\]
Now, let $s_j\nearrow1$ be such that 
\[
\liminf_{s\nearrow1} (1-s)\mathfrak{h}_{s,p}= \lim_{j\to \infty}(1-s_j) \mathfrak{h}_{p,s_j}.
\] 
Now we apply Lemma  \ref{alternative}, with a fixed non-negative function $\varphi \in C_0^{\infty}(\mathbb R^{N}\setminus\{0\})$ such that $\varphi\not=0$ and $u$ defined above. From the equation satisfied by $u$, we then obtain
\begin{equation}
\label{chain-ineq}
\begin{split}
 K_{p,N} \mathfrak{h}_{p}\ \int_{\mathbb R^N} \frac{u^{p-1}}{|x|^{p}}\varphi\, dx& = K_{p,N} \int_{ \mathbb R^N} |\nabla u|^{p-2}  \nabla u \cdot  \nabla \varphi\, dx\\
& = \lim_{j\to \infty}  (1-s_j) \iint_{\mathbb R^N\times \mathbb R^N} \frac{J_p(u(x)-u(y))\,(\varphi(x)-\varphi(y))}{|x-y|^{N+s_j p}}\,dx\,dy\\
& = \lim_{j\to \infty}  \left(  (1-s_j)  \mathcal{C}_{p,s_j}\int_{\mathbb R^N} \frac{u^{p-1}}{|x|^{s_jp}} \varphi\, dx\right).\\
\end{split}
\end{equation}
If we now apply \eqref{louerbaund} with $s_j$ in the place of $s$, \eqref{chain-ineq} implies that
\[
 K_{p,N} \mathfrak{h}_{p}\ \int_{\mathbb R^N} \frac{u^{p-1}}{|x|^{p}}\varphi\, dx\le  \lim_{j\to \infty} (1-s_j)  \mathfrak{h}_{p,s_j} \int_{\mathbb R^N} \frac{u^{p-1}}{|x|^{s_jp}} \varphi\, dx= \liminf_{s\nearrow1} (1-s)\mathfrak{h}_{s,p}\int_{\mathbb R^N}  \frac{u^{p-1}}{|x|^{p}}\varphi\, dx.
\]
Hence, the desired conclusion follows, by canceling the common factor. 
\end{proof}

\begin{proof} [Proof of Theorem \ref{s-to-1}]  By applying Lemma  \ref{limsup},  Theorem \ref{main} and Lemma \ref{liminf},  we have that for every open set $\Omega\subsetneq \mathbb R^N$ it holds
\[{K_{p,N}}\mathfrak{h}_{p}(\Omega)\geq  \limsup_{s\nearrow 1 }(1-s)\, \mathfrak{h}_{s,p}(\Omega)\geq   \liminf_{s\nearrow 1 }(1-s)\, \mathfrak{h}_{s,p}\geq  {K_{p,N}} \mathfrak{h}_{p}
\]

When $\mathfrak{h}_{p}(\Omega)=\mathfrak{h}_{p}$, this implies that  
$$\lim_{s\nearrow 1 }(1-s)\, \mathfrak{h}_{s,p}(\Omega)=\mathfrak{h}_{p}.$$

\end{proof}

\begin{remark}
Actually, with the same proof, one can prove the analogue of Theorem \ref{s-to-1} for convex sets. In other words, if $1<p<\infty$ and $\Omega\subseteq \mathbb R^N$ is a convex set, then we can obtain
\begin{equation}
		\label{convex}\lim_{s\nearrow 1 }(1-s)\, \mathfrak{h}_{s,p}(\Omega)= K_{p,N} \mathfrak{h}_{p}(\Omega)=K_{p,N}\bigg(\frac{p-1}{p}\bigg)^p.
		\end{equation}
This is possible since, for fixed $p>1$,  we have that $sp\geq 1$ when $s$ is sufficiently close to $1$ and, in this range,  for any convex set $\Omega$ it holds
	$$
	\mathfrak{h}_{s,p}(\Omega)=\mathfrak{h}_{s,p}(\mathbb H^N_+),\qquad \text{where}\  
	\mathbb H_N^+:= \mathbb R^{N-1}\times (0,+\infty),
	$$ 
see   \cite[Theorems 6.3]{BBZ}. Moreover, in the specific case of the half-space $\mathbb{H}^N_+$, for $s$ sufficiently close to $1$,  by \cite[Theorem 5.2]{BBZ}, we have that $u(x)=  |x|^{(p-1)/p}$ is a positive local weak solution to the  equation  
\[
(-\Delta_p)^s V=C_{p,s} \frac{V^{p-1}}{d_{\Omega}^{sp}},\qquad \text{in}\ \mathbb H_N^+,
\]
with a suitable positive costant $C_{p,s}$. By using   again formula \eqref{dual-super}, we obtain that
 \[\mathfrak{h}_{s,p} (\mathbb H_N^+)\geq  C_{p,s}.\]
Moreover, by direct verification we see that such a function $u$ is also a (actually classical) positive solution of the equation
\[
-\Delta_p V=\left(\frac{p-1}{p}\right)^p\,\frac{V^{p-1}}{d_{\Omega}^{p}},\qquad \text{in}\ \mathbb H_N^+.
\]
Hence, by using such a function $u$ and  arguing exactly as in \eqref{chain-ineq}, we can  conclude that 
\[
\liminf_{s\nearrow 1}(1-s)\mathfrak{h}_{s,p}(\Omega)=\liminf_{s\nearrow1}(1-s) \mathfrak{h}_{s,p}(\mathbb H_N^+)\geq K_{p,N}   \mathfrak{h}_{p}(\mathbb H_N^+)= K_{p,N}  \left(\frac {p-1}{p}\right)^p.
\]
The  $\limsup$ inequality is provided by  Lemma \ref{limsup}. Hence  the limit \eqref{convex} follows.
 	\end{remark}

	\subsection{The case $p\nearrow \infty$}
	In this subsection we show the following theorem.
	 \begin{theorem}\label{pinfty}
Let  $0<s\leq 1$  and  let $\Omega\subsetneq \mathbb{R}^N$  be an open set.  Then 
\begin{equation}
\label{hardyinfty}
\lim_{p\to \infty} (\mathfrak{h}_{s,p}(\Omega))^{\frac{1}{p}}=1.
\end{equation}
\end{theorem} 
\begin{proof} 
The case $s=1$ is contained in \cite[Theorem 4.4]{BPZ}.
In order to show the limit  in \eqref{hardyinfty} for $\mathfrak{h}_{s,p}(\Omega)$ when $0<s<1$, first we show that the $\limsup$ is smaller than or equal to $1$. To this aim, it is sufficient to use a suitable test function. For every $x_0\in \Omega$, take $r<d_\Omega(x_0)$ and define 
\begin{equation}
\label{trial}
\varphi_\varepsilon(x)=\big( \varepsilon+(r-|x-x_0|)_+   \big)^s -\varepsilon^s, \qquad \mbox{ for } x\in \mathbb{R}^N.
\end{equation}
Since $\varphi_\varepsilon\in C^{0,1}(\mathbb{R}^N)$ and vanishes on $\mathbb{R}^N\setminus B_r(x_0)$, we have that $\varphi_\varepsilon\in \widetilde{W}^{s,p}_0(B_r(x_0))\subseteq \widetilde{W}_{0}^{s,p}(\Omega)$.
Thus, by recalling \eqref{inf-W}, we have that
\[
(\mathfrak{h}_{s,p}(\Omega))^\frac{1}{p}\le 
\frac{[\varphi_\varepsilon]_{W^{s,p}(\mathbb{R}^N)}}{\displaystyle\left\|\frac{ \varphi_\varepsilon}{d_\Omega^s}\right\|_{L^p(\Omega)}}.
\]
By sending $p$ to $\infty$ and by using  \cite[Lemma 2.4]{BPSk}, we obtain that
\begin{equation}\label{limsupinfty}
\begin{split}
\limsup_{p\to\infty} \Big(\mathfrak{h}_{s,p}(\Omega)\Big)^\frac{1}{p}& \leq \limsup_{p\to\infty} \frac{[\varphi_\varepsilon]_{W^{s,p}(\mathbb R^N)}}{\displaystyle\left\|\frac{ \varphi_\varepsilon}{d_\Omega^s}\right\|_{L^p(\Omega)}}
=\frac{[\varphi_\varepsilon]_{C^{0,s}(\mathbb R^N)}}{\displaystyle\left\|\frac{ \varphi_\varepsilon}{d_\Omega^s}\right\|_{L^\infty(B_r(x_0))}}.\\
\end{split}
\end{equation}
We now observe that for every $x,y\in\mathbb{R}^N$
\[
\begin{split}
|\varphi_\varepsilon(x)-\varphi_\varepsilon(y)|&=\left|\big(\varepsilon+(r-|x-x_0|)_+ \big)^s-\big(\varepsilon+(r-|y-x_0|)_+ \big)^s\right|\\
&\le |(r-|x-x_0|)_+ - (r-|y-x_0|)_+ |^s\\
&\le \big||x-x_0|-|y-x_0|\big|^s\le |x-y|^s,
\end{split}
\]
which shows that
\begin{equation}
\label{sholder}
[\varphi_\varepsilon]_{C^{0,s}(\mathbb{R}^N)}\le 1.
\end{equation}
Thus, from \eqref{limsupinfty} we get
\[
\begin{split}
\limsup_{p\to\infty} \Big(\mathfrak{h}_{s,p}(\Omega)\Big)^\frac{1}{p}\le \frac{1}{\displaystyle\left\|\frac{ \varphi_\varepsilon}{d_\Omega^s}\right\|_{L^\infty(B_r(x_0))}}&=\inf_{x\in B_r(x_0)} \frac{d_\Omega(x)^s}{\big( \varepsilon+(r-|x-x_0|)_+\big)^s -\varepsilon^s}\\
&\le \frac{d_\Omega(x_0)^s}{(r+\varepsilon)^s -\varepsilon^s}.
\end{split}
\]
By first taking the limit as $\varepsilon$ goes to $0$ and then as $r$ goes to $d_\Omega(x_0)$, we finally obtain
\[
\limsup_{p\to\infty} \Big(\mathfrak{h}_{s,p}(\Omega)\Big)^\frac{1}{p}\le 1.
\]
%

\noindent If we show  that 
\begin{equation}\label{liminf-h}
	\liminf_{p\to\infty} \left(\mathfrak{h}_{s,p}\right)^\frac{1}{p} \geq  1, \end{equation}
in view of  Theorem \ref{main} and of the previous $\limsup$, we obtain the desired conclusion \eqref{hardyinfty}. 
\par
We recall that
	\[
	\mathfrak{h}_{s,p}=2\,\int_0^1 r^{sp-1}\,\left|1-r^\frac{N-sp}{p}\right|^p\,\Phi_{N,s,p}(r)\,dr>0,
	\]
	where, for every $0<r<1$, the quantity $\Phi_{N,s,p}(r)$ is given by
	\begin{equation*}
		\label{phistrange}
		\Phi_{N,s,p}(r)=\begin{cases}|\mathbb{S}^{N-2}|\,\displaystyle  \int_{-1}^1 \frac{(1-t^2)^\frac{N-3}{2}}{(1-2\,t\,r+r^2)^\frac{N+sp}{2}}\,dt,& \mbox{ if  }N\ge 2,\\
		\displaystyle \frac{1}{(1-r)^{1+sp}}+\frac{1}{(1+r)^{1+sp}},& \mbox{ if }N= 1.
		\end{cases}
	\end{equation*}
	By a simple computation, one can see that 
	\begin{equation}\label{h}
		\mathfrak{h}_{s,p}=2\,\int_0^1 r^{N-1}\,\left(1-r^\frac{sp-N}{p}\right)^p\,\Phi_{N,s,p}(r)\,dr.
	\end{equation}
In the case $N\geq 2$, we observe that, for any $r\in (0,1)$ and for any $t\ge 1/2$, one has 
	$$1-2tr+r^2=1+r(r-2t)\le 1+r(r-1)\le 1.$$
Hence,
	\[\Phi_{N,s,p}(r)\ge|\mathbb{S}^{N-2}|\,\int_{1/2}^1 {(1-t^2)^\frac{N-3}{2}}\,dt.\]
	In the case $N=1$ we have that $$ \Phi_{N,s,p}(r)\geq  \Phi_{N,s,p}(0)=2.$$
	Thus, \eqref{h} implies  that
	\begin{equation*}
		\big(	\mathfrak{h}_{s,p}\big)^{\frac{1}{p}}\ge C_N^{\frac{1}{p}} \left(\int_0^1 r^{N-1}\,\left(1-r^\frac{sp-N}{p}\right)^p\,dr\right)^{\frac{1}{p}},
	\end{equation*}
	where we have set
\[ C_N:= \begin{cases} 2|\mathbb{S}^{N-2}|\,  \int_{1/2}^1 {(1-t^2)^\frac{N-3}{2}}\,dt ,& \hbox{ if  }N\ge 2\\
	2 & \hbox{ if }N= 1.
		\end{cases}
\]
Since, $C_N^{\frac{1}{p}}\rightarrow 1$ as $p\nearrow \infty$, in order to prove \eqref{liminf-h},  it is sufficient to show that, for every $N\geq 1$, it holds
	\begin{equation}\label{done}\liminf_{p\rightarrow \infty}\left(\int_0^1 r^{N-1}\,\left(1-r^\frac{sp-N}{p}\right)^p\,dr\right)^{\frac{1}{p}}\ge 1.
	\end{equation}
	In order to do that, we observe that, for any $r\in (0,1)$ and for any $p\ge p_0$, with $p_0$ fixed such that $p_0\,s>N$, it holds $r^{s-N/p}\le r^{s-N/p_0}$, and thus
	\[
	\left(\int_0^1 r^{N-1}\,\left(1-r^\frac{sp-N}{p}\right)^p\,dr\right)^{\frac{1}{p}}\ge \left(\int_0^1 r^{N-1}\,\left(1-r^{s-\frac{N}{p_0}}\right)^p\,dr\right)^{\frac{1}{p}}=\left\|1-r^{s-\frac{N}{p_0}}\right\|_{L^p_{\mu((0,1))}},
	\]
where we have denoted by $\|\cdot\|_{L^p_{\mu((0,1))}}$ the $L^p$-norm with respect to the measure $d\mu=r^{N-1}\,dr$.
Finally, taking the limit as $p\nearrow \infty$, we deduce that
$$
\liminf_{p\rightarrow \infty}\left(\int_0^1 r^{N-1}\,\left(1-r^\frac{sp-N}{p}\right)^p\,dr\right)^{\frac{1}{p}}\ge \sup_{r\in (0,1)}|1-r^{s-\frac{N}{p_0}}|=1.
$$
This shows \eqref{done}, thus concluding the proof of \eqref{liminf-h}.
\end{proof}

\section{A Cheeger type inequality}
 In the main theorem of this section,  we  provide a  lower bound for $\lambda_{s,p}(\Omega)$ given by \eqref{speigenv},  in terms of the classical and fractional {\it Cheeger constants} 
$h_{1}(\Omega)$ and  $h_{s}(\Omega)$. We recall that
 \[
h_{s}(\Omega)=\inf\bigg\{\frac{P_s(E)}{|E|}\ :\ E\Subset\Omega \ \text{smooth},\, |E|>0\bigg\}
\]
where   $$P_s(E)=[1_E]_{W^{s,1}(\mathbb{R}^N)}= \iint_{\mathbb{R}^N\times \mathbb{R}^N} \frac{|1_E(x)-1_E(y)|}{|x-y|^{N+s}}\,dx\,dy$$ is the nonlocal $s$-perimeter of $E$.
We explicitly note that our result covers also the case $s=1$.
\begin{theorem}\label{cheeger} Let $0<s\leq 1$ and $sp>N$.  Let $\Omega\subsetneq \mathbb{R}^N$  be an open set and define
\[
\lambda_{s,p}(\Omega)=\inf_{u\in C^\infty_0(\Omega)}\left\{[u]^p_{W^{s,p}(\mathbb{R}^N)} :\, \int_{\Omega} |u|^p \, dx=1\right\}.
\]
If $r_{\Omega}<+\infty$  then it holds
\begin{equation}
\label{hardy-cheeger2}   
\lambda_{s,p}(\Omega)\geq \mathfrak{h}_{s,p}\, \left(\frac{h_1(\Omega)}{N}\right)^{sp},
\end{equation}
where $h_1(\Omega)$ is defined by \eqref{cheegcon}. In particular,  for $0<s<1$ we also have
\begin{equation}\label{cheeger-s}  \lambda_{s,p}(\Omega)\geq \frac{\mathfrak{h}_{s,p}}{N^{sp}}  \left(\frac{(1-s)\,s}{2N \omega_N} h_s(\Omega)\right)^p.
\end{equation}

 \end{theorem}
\begin{proof}
Let us suppose that $r_{\Omega}<+\infty$. 
Since $sp>N$,  thanks to Theorem  \ref{main}, we have that $\Omega$ satisfies the Hardy inequality with  $\mathfrak{h}_{s,p}(\Omega)\geq \mathfrak{h}_{s,p}>0$. 
Then

			\[ \int_{\Omega} |u|^p \, dx \le r^{sp}_{\Omega} \int_{\Omega} \frac{|u|^p}{d_{\Omega}^{sp}} \, dx \le \frac{1}{\mathfrak{h}_{s,p}(\Omega)} \, r_{\Omega}^{sp}\,[u]^p _{W^{s,p}(\mathbb R^N)}\, \quad \mbox{ for every } u\in C^{\infty}_0(\Omega). 
			\]
			By taking the infimum on $C^{\infty}_0(\Omega)$, we easily get  \begin{equation}\label{hardy-inradius}
			\lambda_{s,p}(\Omega)\geq \frac{\mathfrak{h}_{s,p}(\Omega)}{r^{sp}_{\Omega}}\geq \frac{\mathfrak{h}_{s,p}}{r^{sp}_{\Omega}}.
\end{equation}
The above estimate,  combined with the well  known inequality
 \begin{equation}\label{inr-chee} h_1(\Omega)\leq  \frac N {r_{\Omega}}, \end{equation}
 gives \eqref{hardy-cheeger2}.
 In order to show \eqref{cheeger-s}, it is sufficient to note  that, thanks to \cite[Corollary 4.4]{BLP}, for every open bounded set $E\Subset\Omega$ with smooth boundary, it holds
\[
\left(\frac{P(E)}{|E|}\right)^s\geq \frac{P_s(E)}{|E|}  \frac{(1-s)s}{2N \omega_N} .
\]
 Hence, by arbitrariness of $E$ we get  
 \[
( h_1(\Omega))^s\geq  \frac{(1-s)\,s}{2N \omega_N}  \inf\bigg\{\frac{P_s(E)}{|E|}\ :\ E\Subset\Omega\ \text{smooth},\, |E|>0\bigg\}= \frac{(1-s)\,s}{2N \omega_N}   h_s(\Omega).
 \]
Joining  \eqref{hardy-cheeger2} with the above inequality,   we obtain
\eqref{cheeger-s}.
	\end{proof}
			
		\begin{remark}\label{cheegerimpr}	
We note that our  estimate   \eqref{hardy-cheeger2}   appears to be  new already in the case $s=1$, where  it  improves (for $p>N$)
 the 
celebrated Cheeger inequality
\begin{equation}\label{cheeg} \lambda_{p}(\Omega)\geq \left ( \frac{h_1(\Omega)}{p}   \right)^p,
\end{equation}
valid for every $1<p<\infty$ and for every open set $\Omega$ (for a proof, see \cite{LW,KF}).
Indeed, when $s=1$, by joining \eqref{cheeg} and \eqref{hardy-cheeger2}, we now get 
\begin{equation*} \lambda_{p}(\Omega)\geq  \max \left\{\left ( \frac{p-N}{N} \right)^p  , 1\right\}\, \left(\frac{h_1(\Omega)}{p}\right)^p,
\end{equation*}
which holds for every open set $\Omega\subseteq\mathbb{R}^N$.
The main interest of this results is that this is stable as $p$ goes $\infty$, i.e. we have
\[
\lim_{p\nearrow \infty} \left(\max \left\{\left ( \frac{p-N}{N} \right)^p  , 1\right\}\, \left(\frac{h_1(\Omega)}{p}\right)^p\right)^\frac{1}{p}=\frac{h_1(\Omega)}{N},
\] 
while the right-hand side of \eqref{cheeg} raised to the power $1/p$  converges to $0$. 
\end{remark}
By combining the estimate \eqref{hardy-inradius} with the asymptotic behaviour of $(\mathfrak{h}_{s,p}(\Omega))^{1/p}$ as $p\to \infty$, 
 we get the next result, which clarifies the interest of Remark \ref{cheegerimpr}.
\begin{proposition}\label{convergenzaautov} Let $0<s\leq 1$ and let $\Omega\subsetneq \mathbb{R}^N$  be an open set. 
		Then\footnote{It is intended that $1/r_\Omega^s=0$, in the case $r_\Omega=+\infty$.} 
		\begin{equation}
		\label{peigenlimit}
		\lim_{p\to \infty}(\lambda_{s,p}(\Omega))^\frac{1}{p}=\frac 1 {r_{\Omega}^s}=\lambda_{s,\infty}(\Omega),
 \end{equation}
where $\lambda_{s,\infty}(\Omega)$ is defined through the following  minimization problem  
\[
\lambda_{s,\infty}(\Omega)=
\inf_{u\in C^{0,s}(\overline\Omega)}\left\{[u]_{C^{0,s}(\overline\Omega)} :\,  \|u\|_{L^\infty(\Omega)}=1,\ u=0\ \text{on}\ \partial\Omega\right\}.
\]
Moreover, when $r_\Omega<+\infty$, a minimizer of the last problem is given by
\[
U=\left(\frac{d_\Omega}{r_\Omega}\right)^s.
\]
\end{proposition}


 \begin{proof}
 First of all,  we prove that 
			 \begin{equation}
			 \label{2dis}
			 \limsup_{p\to\infty} \Big(\lambda_{s,p}(\Omega)\Big)^\frac{1}{p}\leq \frac{1}{r_\Omega^s}.
			 \end{equation}
We take $r<r_\Omega$, thus there exists $x_0\in\Omega$ such that $B_r(x_0)\subseteq \Omega$.			 
For every $\varepsilon>0$, we take the same function $\varphi_\varepsilon$ defined in \eqref{trial}.
This implies that
\[
\Big(\lambda_{s,p}(\Omega)\Big)^\frac{1}{p}\le 
\frac{[\varphi_\varepsilon]_{W^{s,p}(\mathbb{R}^N)}}{\displaystyle\left\|  \varphi_\varepsilon \right\|_{L^p(\Omega)}}.
\]
By sending $p$ to $\infty$ (and by using  \cite[Lemma 2.4]{BPSk} when $0<s<1$),  we obtain that 
\[
\begin{split}
\limsup_{p\to\infty} \Big(\lambda_{s,p}(\Omega)\Big)^\frac{1}{p}& \leq \limsup_{p\to\infty} \frac{[\varphi_\varepsilon]_{W^{s,p}(\mathbb R^N)}}{\|  \varphi_\varepsilon \|_{L^p(\Omega)}}
=\begin{cases}\displaystyle \frac{[\varphi_\varepsilon ]_{C^{0,s}(\mathbb R^N)}}{\|  \varphi_\varepsilon \|_{L^\infty(\Omega)}}, \quad \hbox{ if } 0<s<1,\\ 
\\
\displaystyle \frac{\|\nabla \varphi_\varepsilon\|_{L^{\infty}(\mathbb{R}^N)}}{\|  \varphi_\varepsilon \|_{L^\infty(\Omega)}},  \quad \hbox{ if } s=1.
\end{cases}
\end{split}
\]
By observing that 
\[
\|\nabla \varphi_\varepsilon\|_{L^{\infty}(\mathbb{R}^N)}=[\varphi_\varepsilon]_{C^{0,1}(\mathbb{R}^N)},
\]
we thus obtain
\begin{equation}
\label{ap}
\limsup_{p\to\infty} \Big(\lambda_{s,p}(\Omega)\Big)^\frac{1}{p}\le \frac{[\varphi_\varepsilon]_{C^{0,s}(\mathbb R^N)}}{\|  \varphi_\varepsilon \|_{L^\infty(\Omega)}},\qquad \text{for}\ 0<s\le 1.
\end{equation}
We now observe that
\[
\|\varphi_\varepsilon\|_{L^\infty(\Omega)}=(r+\varepsilon)^s-\varepsilon^s.
\]
By recalling also \eqref{sholder}, from \eqref{ap} we thus obtain for every $\varepsilon>0$
\[
\limsup_{p\to\infty} \Big(\lambda_{s,p}(\Omega)\Big)^\frac{1}{p}\le \frac{1}{(r+\varepsilon)^s-\varepsilon^s},\qquad \text{for}\ 0<s\le 1.
\]
By taking the limit as $\varepsilon$ goes to $0$ and using the arbitrariness of $r<r_\Omega$, we get \eqref{2dis}.
In particular, if $r_{\Omega}=+\infty$, \eqref{2dis} implies that  
 \[\lim_{p\to \infty}(\lambda_{s,p}(\Omega))^\frac{1}{p}=0=\frac 1 {r_{\Omega}^s}.\]
 In the case when $r_{\Omega}<\infty$, for every $0<s\leq 1$, we can  use  \eqref{hardy-inradius} and \eqref{hardyinfty}, to obtain that
			 \begin{equation*}
			 \liminf_{p\to \infty}(\lambda_{s,p}(\Omega))^\frac{1}{p}\geq \liminf_{p\to \infty} \frac{(\mathfrak{h}_{s,p}(\Omega))^\frac{1}{p}}{r^{s}_{\Omega}}= \frac 1 {r_{\Omega}^s}.			 \end{equation*}	 
	\vskip.2cm\noindent
We now prove that 
\begin{equation}
\label{problemasup}
\frac{1}{r_{\Omega}^s}=\lambda_{s,\infty}(\Omega).
\end{equation}
Let $\varphi\in C^{0,s}(\overline\Omega)$ be admissible for the problem defining $\lambda_{s,\infty}(\Omega)$. For every $x\in\Omega$, we take $y_x\in\partial\Omega$ such that $d_\Omega(x)=|x-y_x|$. Thus, we have
\[
|\varphi(x)|=|\varphi(x)-\varphi(y_x)|\le |x-y_x|^s\,[\varphi]_{C^{0,s}(\overline\Omega)}=d_\Omega(x)^s\,[\varphi]_{C^{0,s}(\overline\Omega)}\le r_\Omega^s\,[\varphi]_{C^{0,s}(\overline\Omega)}.
\]
This shows that 
\[
\frac{1}{r_\Omega^s}\le [\varphi]_{C^{0,s}(\overline\Omega)},
\]
thanks to the normalization on $\varphi$. It is intended that the left-hand side is zero, in the case $r_\Omega=+\infty$.
The previous inequality in turn implies that
\[
\lambda_{s,\infty}(\Omega)\ge \frac{1}{r_\Omega^s}.
\]
Finally, if $r_\Omega<+\infty$ we take the function 
\[
\varphi=\frac{d_\Omega^s}{r_\Omega^s},
\]
which is admissible for $\lambda_{s,\infty}(\Omega)$. This gives
\[
\lambda_{s,\infty}(\Omega)=\frac{1}{r_\Omega^s}\,[d_\Omega^s]_{C^{0,s}(\overline\Omega)}\le \frac{1}{r_\Omega^s},
\]
thanks to the fact that $d_\Omega^s$ is $s-$H\"older continuous, with H\"older constant less than or equal to $1$. This shows \eqref{problemasup} and that $ {r_{\Omega}^{-s}} d_\Omega^s  $  is a minimizer for the problem defining $\lambda_{s,\infty}(\Omega)$.
\par
In the case $r_\Omega=+\infty$, it is sufficient to take $M>0$ and use the test function 
\[
\varphi_M=\frac{\min\{d_\Omega^s,M^s\}}{M^s}.
\]
This would give 
\[
\lambda_{s,\infty}(\Omega)\le \lim_{M\to\infty}\frac{1}{M^s}\,\big[\min\{d_\Omega^s,M^s\}\big]_{C^{0,s}(\overline\Omega)}=\lim_{M\to\infty}\frac{1}{M^s}\,[d_\Omega^s]_{C^{0,s}(\overline\Omega)}\le \lim_{M\to\infty} \frac{1}{M^s}=0,
\]
thus proving  \eqref{problemasup} in the case $r_\Omega=+\infty$, as well.
\end{proof}
\begin{remark}
We recall that, when  $s=1$,  the limit  
\[\lim_{p\to \infty}(\lambda_{p}(\Omega))^{1/p} =\frac 1 {r_{\Omega}}\]
 has been shown  in \cite[Theorem 3.1]{FIN} and \cite[Lemma 1.2]{JLM}, when $\Omega\subseteq\mathbb{R}^N$ is a bounded open set. 
 Later this result has been extended to every open  set in  \cite[Corollary 6.1]{BPZ}. 
\end{remark}


\begin{remark} 
Thanks to the previous result, we can observe that the lower bound \eqref{hardy-cheeger2} becomes sharp in the limit, as $p$ goes to $\infty$. Indeed, for every $0<s\leq 1$ and for every open set $\Omega\subseteq\mathbb R^N$, by combining \eqref{peigenlimit} and \eqref{inr-chee}, we get the following inequality  
\begin{equation}
\label{cheegerbelow}
 \lambda_{s,\infty}(\Omega)= \lim_{p\to \infty}(\lambda_{s,p}(\Omega))^\frac{1}{p} 
 \geq \left(\frac{h_1(\Omega)}{N}\right)^s.
\end{equation}
Such an estimate {\it is sharp}, since it becomes an identity when $\Omega=B_R(x_0)$, thanks to the fact that 
$$
\lambda_{s,\infty}(B_R(x_0))=\frac{1}{R^s}=\left(\frac{h_1(B_R(x_0))}{N}\right)^s.
$$

%
 \end{remark}

\end{document}